\documentclass[10pt]{amsart}
\usepackage{amscd,amssymb,color}
\usepackage[all]{xy}
\usepackage{xcolor}

\begin{document}

\newtheorem{prop}{Proposition}[section]
\newtheorem{thm}[prop]{Theorem}
\newtheorem{lemma}[prop]{Lemma}
\newtheorem{cor}[prop]{Corollary}
\newtheorem{Question}[prop]{Question}
\theoremstyle{definition}
\newtheorem{Example}[prop]{Example}
\newtheorem{Examples}[prop]{Examples}
\newtheorem{Remark}[prop]{Remark}

\newcommand{\extto}{\xrightarrow}
\newcommand{\dpp}{\prime\prime}
\newcommand{\br}{{\bf r}}
\newcommand{\bx}{$_{\fbox{}}$\vskip .2in}
\newcommand{\Gr}{\operatorname{Gr}\nolimits}
\newcommand{\gr}{\operatorname{gr}\nolimits}
\newcommand{\Mod}{\operatorname{Mod}\nolimits}
\newcommand{\md}{\operatorname{mod}}
\newcommand{\End}{\mbox{End}}
\newcommand{\Hom}{\mbox{Hom}}
\renewcommand{\Im}{\operatorname{Im}\nolimits}
\newcommand{\pd}{\mbox{pd}}
\newcommand{\Ker}{\operatorname{Ker}\nolimits}
\newcommand{\Coker}{\mbox{Coker}}
\newcommand{\coker}{\mbox{coker}}
\newcommand{\Tor}{\mbox{Tor}}
\renewcommand{\dim}{\mbox{dim}}
\newcommand{\gldim}{\mbox{gl.dim}}
\newcommand{\Ext}{\operatorname{Ext}\nolimits}
\newcommand{\HH}{\operatorname{HH}}
\newcommand{\op}{^{\operatorname{op}}}
\newcommand{\ev}{{\operatorname{ev}}}
\newcommand{\pr}{^{\prime}}
\newcommand{\f}{\operatorname{fin}}
\newcommand{\Sy}{\operatorname{syz}}
\newcommand{\semi}{\mathbin{\vcenter{\hbox{$\scriptscriptstyle|$}}
\;\!\!\!\times }}
\newcommand{\fralg}{K\!\!<\!\!x_1,\dots,x_n\!\!>}
\newcommand{\mo}{\mathfrak{o}}
\newcommand{\mt}{\mathfrak{t}}
\newcommand{\Span}{\operatorname{Span}\nolimits}
\newcommand{\cN}{\mathcal N}
\newcommand{\rrad}{\mathfrak{r}}

\title{On the diagonal subalgebra of an Ext algebra}

\author{E. L.\ Green}
\address{Department of Mathematics\\
Virginia Tech\\
Blacksburg, VA 24061-0123\\ USA}
\email{green@math.vt.edu}
\author{N. Snashall}
\thanks{The second author was supported by an LMS scheme 4 grant.}
\address{Department of Mathematics\\
University of Leicester\\
Leicester LE1 7RH\\ UK}
\email{njs5@le.ac.uk}
\author{\O. Solberg}
\address{Department of Mathematical Sciences\\
NTNU\\
N-7491 Trondheim\\ Norway}
\email{oyvinso@math.ntnu.no}
\thanks{The third author is supported by ``Clusters, combinatorics and computations in algebra'' grant no.\ 23100
  from the Research Council of Norway}
\author{D. Zacharia}
\thanks{The fourth author was partially supported by the NSA grant
  H98230-11-1-0152.}
\address{Department of Mathematics\\Syracuse University\\Syracuse, NY
13244\\USA}
\email{zacharia@syr.edu}



\begin{abstract} Let $R$ be a Koszul algebra over a field $\Bbbk$ and
  $M$ be a linear $R$-module.  We study a graded subalgebra $\Delta_M$
  of the Ext-algebra $\Ext_R^*(M,M)$ called the diagonal subalgebra
  and its properties.  Applications to the Hochschild cohomology ring
  of $R$ and to periodicity of linear modules are given. Viewing $R$
  as a linear module over its enveloping algebra, we also show that
  $\Delta_R$ is isomorphic to the graded center of the Koszul dual of
  $R$.
\end{abstract}

\date{\today}

\maketitle
\section{Introduction}


Let $\Bbbk$ be a field and let $R=R_0\oplus R_1\oplus\cdots$ be a
Koszul $\Bbbk$-algebra, with $R_0=\Bbbk\times\cdots\times\Bbbk$ and
with graded radical $\mathfrak{ r}=\bigoplus_{i>0} R_i$. Let $M$ be a
linear $R$-module and let $\Ext^*_R(M,M)$ be its Ext (or Yoneda)
algebra.  It is well-known that $\Ext^*_R(M,M)$ has a bigrading
induced from the homological grading and from the internal grading of
$M$.  The main goal of this paper is to study a particular subalgebra
$\Delta_M$, of $\Ext^*_R(M,M)$ called the \emph{diagonal subalgebra}.
Namely, this subalgebra will be generated by all the elements of
$\Ext^*_R(M,M)$ of bidegree $(i,-i)$ for $i\geqslant 0$.  We study the
graded algebra decomposition
\[\Ext^*_R(M,M)= \Delta_M\oplus \cN_M\]
of the Ext-algebra given in Proposition \ref{decomposition},
so viewing $R$ as a linear module over its enveloping
  algebra, we obtain
  \[ \HH^*(R)=\Delta_R\oplus \mathcal N_R,\]
  where $\HH^*(R)$ denotes the Hochschild cohomology ring of $R$.  We
  show that
\begin{enumerate}
\item If $M$ is a linear $R$-module of finite graded length, then the
  two sided ideal $\mathcal N_M$ consists of nilpotent elements.
\item If $Z_{\gr}(\mathcal E_R)$ denotes the graded center of the
  Koszul dual of $R$, that is, of $\Ext^*_R(R/\mathfrak r,R/\mathfrak
  r)$, then \[\Delta_R\cong Z_{\gr}(\mathcal E_R).\]
\item We use properties of $\Delta_M$ to obtain a characterization of
  the linear $R$-modules $M$ that are eventually periodic over a Koszul
  algebra.
\item In addition, we prove in the last section that if $R$ is a
  selfinjective algebra over an algebraically closed field and
  satisfies the {\bf Fg} finiteness condition, then, the property that
  all the simple modules are periodic is equivalent to the algebra
  being periodic when viewed as a module over its enveloping algebra.
\end{enumerate}

\smallskip

\noindent Throughout this paper, we only consider graded $R$-modules
that have graded projective resolutions, such that each projective
module occurring in the resolution is finitely generated. It is
well-known that minimal graded resolutions exist in our situation. We
start by recalling some of the basic notation that will be used in
this paper.  Let $M$ and $N$ be two graded $R$-modules. By a
homomorphism of degree $i$ we mean a module homomorphism from $M$ to
$N$ taking $M_k$ into $N_{i+k}$ for each integer $k$. By abuse of
language we will use the term ``graded homomorphism" for degree zero
homomorphisms. Let $M$ be a finitely generated graded $R$-module and
let $m\in\mathbb Z$. Then $M(m)$ will denote the graded shift of the
module $M$, that is the graded module whose $i$-th graded piece is
$M(m)_i=M_{i+m}$. Similarly, if $(P^{\bullet}_M, d^{\bullet}_M)$ is a
graded projective resolution of $M$, then its graded shift
$(P^{\bullet}_M(m), d^{\bullet}_M(m))$ is defined in the obvious way.

\smallskip

\noindent Denote by $\Gamma$, the Ext-algebra of $M$, that
is $$\Gamma=\Ext_R^*(M,M)=\bigoplus_{n\geqslant 0}\Ext_R^n(M,M).$$
\noindent It is well-known that $\Gamma$ is an associative
$\Bbbk$-algebra with the multiplication given by the Yoneda
product. For the convenience of the reader we recall one way of
looking at this product. Let
$$P_M^\bullet\space\ \  \cdots \extto{}
P^{n+1}_M\extto{d_M^{n+1}} P^n_M\extto{d_M^n} \cdots \extto{}
P^1_M\extto{d_M^1} P^0_M \extto{d_M^0} M\to 0$$ be a minimal graded projective
resolution of $M$.  Let $\eta$ be an element of $\Ext_R^n(M,M)$. We
may represent it as a homomorphism $\eta\colon P_M^n\rightarrow M$
having the property that the composition $\eta d_M^{n+1}=0$. Now let
$\xi\colon P_M^m\rightarrow M$ represent an element of
$\Ext_R^m(M,M)$.  For each $i = 0, \dots , m$, we have liftings
$l^i\colon P_M^{n+i}\rightarrow P_M^i$ of $\eta$ and we obtain the
following commutative diagram with exact rows:
$$
\xymatrix{
\cdots\ar[r]& P_M^{n+m} \ar[r]^{d_M^{m+n}} \ar[d]^{l^m} & \cdots\ar[r] & P_M^{n+1} \ar[r]^{d_M^{n+1}}
\ar[d]^{l^1} & P_M^n \ar[d]^{l^0}\ar[dr]^{\eta} & \\\cdots\ar[r]&  P_M^m \ar[r]^{\phantom{xx}d_M^m}
\ar[d]^{\xi}& \cdots\ar[r] & P_M^1 \ar[r]^{\phantom{x}d_M^1} & P_M^0 \ar[r]^{d_M^0} & M
\\ &M& & & & }
$$
\noindent It is well-known that the composition $\xi l^m$ represents
the element $\xi\eta$ of $\Ext_R^{m+n}(M,M)$ and that this
multiplication does not depend on the choice of the liftings $l^i$.
Note also that if $\eta$ is a homomorphism of degree $j$, then we may
assume without loss of generality that each lifting $l^i$ has degree
$j$ as well.

\medskip

\noindent The Ext-algebra of $M$ is a bigraded $\Bbbk$-algebra in the
following way: For each integer $i$, let $\Hom_R(M,M)_i$ denote the
space of all graded homomorphisms of degree $i$ from $M$ to $M$. Then
$\End_R(M,M)$ is a graded $R_0$-module by putting
$\End_R(M,M)=\bigoplus_{i\in\mathbb Z}\Hom_R(M,M)_i$.  Passing to
derived functors, for each $n\geqslant 0$, $\Ext_R^n(M,M)$ has a
grading induced from the grading on $\End_R(M,M)$.  More precisely,
let
\[P_M^\bullet\space\ \  \cdots \extto{}
P_M^{n+1}\extto{d_M^{n+1}} P_M^n\extto{d_M^n} \cdots \extto{}
P_M^1\extto{d_M^1} P_M^0 \extto{d_M^0} M\to 0\]
be a minimal graded projective resolution of $M$.  If $\eta\colon
P_M^n\to M$ represents an element of $\Ext^n_R(M,M)$, then $\eta$ is a
sum of graded homomorphisms of various degrees. Denoting by
$\Ext^n_R(M,M)_i$ the elements of $\Ext^n_R(M,M)$ represented by
elements of $\Hom_R(P_M^n,M)_i$ we obtain the
bigrading \[\Ext^*_R(M,M)=\bigoplus_{n\geqslant
  0}\bigoplus_{i\in\mathbb Z}\Ext^n_R(M,M)_i.\]
\noindent The reader may easily check that the Yoneda product is compatible
with this bigrading.

\smallskip

\noindent Finally, we remark that most of the results in this paper concern Koszul algebras and linear
modules, and we recall that a graded $R$-module generated in degree
zero is called a \emph{linear module} if it has a graded projective
resolution such that for each $n$, the $n$-th term of the resolution
is generated in degree $n$.  So in the case when $M$ is a linear
module, this implies that for each nonnegative integer $n$, we have
$\Ext_R^n(M,M)=\bigoplus_{i\geqslant -n}\Ext_R^n(M,M)_i$.

\section{The Diagonal subalgebra of an Ext-algebra}\label{sec:two}

Throughout this section $R=R_0\oplus R_1\oplus\cdots$ is a Koszul
$\Bbbk$-algebra, where $R_0= \Bbbk\times\cdots\times\Bbbk$, and $M$ is
a linear $R$-module. The \emph{diagonal subalgebra} of $\Ext^*_R(M,M)$
is defined as \[\Delta_M=\bigoplus_{n\geqslant 0}\Ext^n_R(M,M)_{-n}\]
and we set $\mathcal N_M=\bigoplus_{i>-n}\bigoplus_{n\geqslant
  0}\Ext^n_R(M,M)_i$.  For each nonnegative integer $n$ we also define
$\Delta_M^n=\Ext^n_R(M,M)_{-n}$. Note that the Yoneda product induces
a two-sided $\Delta_M$-module structure on $\mathcal N_M$. It is also
clear that $\cN_M$ is an ideal of $\Ext^*_R(M,M)$.  We start with a
very simple example. A less obvious one is presented later in this
section.

\begin{Example} Let $R$ be a Koszul algebra and let $M$ be a
  semisimple linear module. It is easy to show in this case that
  $\Delta_M=\Ext^*_R(M,M)$.
\end{Example}

\noindent We have the following immediate consequence of the
definition:

\begin{prop}\label{decomposition} Let $R$ be a Koszul algebra and let
  $M$ be a linear $R$-module. Then we have a decomposition
  $\Ext^*_R(M,M)=\Delta_M\oplus \cN_M$ and $\cN_M$ is an ideal of
  $\Ext^*_R(M,M)$.
\end{prop}

\noindent The following general results concern modules of finite
graded length.  Recall that a graded module $M$ has \emph{finite
  graded length $d$}, if, for some integers $a\leqslant b$, $M_a\ne
0$, $M_b\ne 0$, $M_i=0$ for $i<a$ and $i>b$, and $d=b-a+1$.

\begin{lemma}\label{iszero} Let $R$ be a Koszul algebra, let $M$ be a linear $R$-module
  of finite graded length $d$, and let $P_M^n$ be the $n$-th term in a
  minimal graded projective resolution of $M$.  If $\eta\colon
  P_M^n\rightarrow M$ represents an element of $\Ext_R^n(M,M)_s$ and
  $s>-n+d$, then $\eta=0$ in $\Ext_R^n(M,M)$.
\end{lemma}

\begin{proof} Note that $M$ lives in degrees $0,\ldots, d-1$. Since
  the top of $P_M^n$ lies in degree at least $n$, every homomorphism
  $P_M^n\to M$ of degree greater than $-n+d$ is zero so the result
  follows.
\end{proof}

\begin{prop}\label{nilpot} Let $R$ be a Koszul algebra and $M$ a
  linear $R$-module of finite graded length.  Then every element of
  $\cN_M$ is nilpotent.
\end{prop}

\begin{proof} Assume that the graded length of $M$ is $d$, and let
  $x\in\cN_M$.  Then $x$ can be written as a finite sum
  $\sum_{k\geqslant 0}\sum_{j>-k}x_{k,j}$ where
  $x_{k,j}\in\Ext^k_R(M,M)_j$.  By abuse of notation, we view
  $x_{k,j}$ as a homomorphism from $P_M^k$ to $M$ of degree $j$
  greater than $-k$ such that $x_{k,j}d_M^{k+1}=0$, where
  $(P_M^{\bullet},d_M^{\bullet})$ is a minimal graded projective
  resolution of $M$.  Consider $x^{d+1}=(\sum_{k\geqslant
    0}\sum_{j>-k}x_{k,j})^{d+1}$.  Each term $y$ in $x^{d+1}$ is a
  product of $d+1$ $x_{k,j}$'s , so as an element in some
  $\Ext_R^n(M,M)$, $y$ is represented by a graded map of degree
  greater or equal to $-n+d+1$.  By Lemma \ref{iszero}, each $y =0$
  and hence $x^{d+1}=0$.
\end{proof}

\medskip

\noindent We have the following slightly more explicit way of
describing the subalgebra $\Delta_M$ and the ideal $\mathcal N_M$ for
a linear module $M$.  First, let $\mathfrak r=R_1\oplus
R_2\oplus\cdots$ denote the graded radical of $R$, and let $(P_M^\bullet,
d_M^\bullet)$ be a linear resolution of $M$. Let $\eta\colon
P_M^n\rightarrow M$ be a $1\times k$ matrix with entries in $M$
representing an element of $\Ext_R^n(M,M)$, where $k$ is the number of
indecomposable summands of $P_M^n$. Then $\eta$ represents an element
of $\Delta_M^n$ if and only if $\eta=[m_1,\ldots,m_k]$ where each
$m_i\in M_0$, or equivalently, $\eta$ is a degree $-n$ map from
$P_M^n$ to $M$. Similarly, $\eta\colon P_M^n\rightarrow M$ represents
an element of $\cN_M$ if and only if $\Im \eta\subseteq M\mathfrak r$
(that is, each $m_i$ is in the radical of $M$).

\medskip

\noindent In view of the preceding results we introduce the following
definitions. Let $R$ be a Koszul algebra and let $M$ be a linear
$R$-module. We say that a nonzero homogeneous element
$\eta\in\Ext_R^n(M,M)$ is {\it strongly radical,} if it is in
$\cN_M$. If in addition, $R$ is finite dimensional, then we call
$\eta$ {\it strongly nilpotent}. Note that it is possible to have a
nilpotent homogeneous element in $\Ext_R^*(M,M)$ that is not strongly
nilpotent. For instance if we look at the polynomial ring in $n$
variables over a field, and we let $M$ be the unique graded simple
module, then $\Ext_R^*(M,M)$ is finite dimensional and in this case
$\Ext_R^*(M,M)=\Delta_M$, so every element of $\Delta_M$ of positive
(homological) degree is nilpotent.

\medskip

\noindent We have the following illustration of Proposition
\ref{decomposition}. Viewing a Koszul algebra $R$ as a module over its
enveloping algebra $R^e=R^{\op}\otimes_{\Bbbk}R$,  the
  module $R$ is linear (see \cite{GHMS} for instance), and we get a
decomposition of the Hochschild cohomology ring as
$$\HH^*(R)=\Delta_R\oplus\mathcal N_R.$$

\medskip

\noindent Let $\mathcal E_R$ denote the Ext-algebra of $R$, that is
$\mathcal E_R=\Ext_R^*(R/\mathfrak r,R/\mathfrak r)$. Recall also that
the {\it graded center} of $\mathcal E_R$ is the graded subring
$Z_{\gr}(\mathcal E_R)$ generated by all the homogeneous elements $u$
such that $uv=(-1)^{|u||v|}vu$ for every homogeneous element $v$ of
$\mathcal E_R$, where $|x|$ denotes the degree of the homogeneous
element $x$. There is a homomorphism $T$ of graded algebras
$$T\colon \HH^*(R)\rightarrow\mathcal E_R$$
given by $T(\eta)=R/\mathfrak r\otimes_R\eta$. Since $R$ is Koszul,
the image of this homomorphism is the graded center $Z_{\gr}(\mathcal E_R)$ (see
\cite{BGSS}). Summarizing, we have the following characterization of
the graded center of a Koszul algebra:

\begin{thm}\label{gr-cent}
  Let $S$ be a Koszul algebra and let $R$ be its Koszul dual, so
  $S=\mathcal E_R$. Let $\HH^*(R)=\Delta_R\oplus\mathcal N_R$. Then
  the homomorphism $T$ induces an isomorphism of graded
  $\Bbbk$-algebras $\Delta_R\cong Z_{\gr}(S)$.
\end{thm}

\begin{proof} Let $(P^{\bullet},d^{\bullet})$ be a linear resolution
  of $R$ over its enveloping algebra $R^e$, and let
  $\varepsilon\in\mathcal N_R$ be homogeneous of degree $n$. Now,
  $\varepsilon$ can be represented by a matrix $[z_1,\dots,z_k]\colon
  P^n\rightarrow R$ where $k$ denotes the number of indecomposable
  summands of $ P^n$ and where each entry $z_i\in\mathfrak r$, so it
  is clear that $T(\varepsilon)=0$. Therefore we have an induced
  homomorphism of graded algebras $T\colon\Delta_R\rightarrow Z_{\gr}(
  \mathcal E_R )$. By \cite{BGSS} it is clear that this restriction of
  $T$ to $\Delta_R$ is surjective.

\medskip
\noindent Now let $Q$ be an indecomposable projective $R^e$-summand of
$ P^n$.  We may write $Q=(e^o_v\otimes_{\Bbbk} e_w)R^e$ where $e^o_v$
and $e_w$ are primitive idempotents in $R^{\op}$ and $R$ respectively.
If we have a nonzero map $\eta\colon\ Q\rightarrow R(-n)$ such that
$\eta(e^o_v\otimes_{\Bbbk} e_w)=r_0\in R_0$, then clearly $R/\mathfrak
r\otimes_R\eta\ne 0$. This clearly implies that the restriction of $T$
to $\Delta_R$ is one-to-one.
\end{proof}

\noindent It has been conjectured in \cite{SS} that for a finite
dimensional algebra $R$, the Hochschild cohomology modulo the ideal
generated by the nilpotent elements is finitely generated as an
algebra over the ground field. This conjecture has been disproved by
Xu (\cite{X}), see also \cite{S}. However, there are many instances
where this quotient is finitely generated, to which we add the
following result:

\begin{prop} Let $R$ be a finite dimensional Koszul $\Bbbk$-algebra
  with Koszul dual $\mathcal E_R$ and let $\widetilde{\mathcal N}$ be
  the ideal of $\HH^*(R)$ generated by the homogeneous nilpotent
  elements. Assume that the graded center of $\mathcal E_R$ is a
  finitely generated $\Bbbk$-algebra.  Then
  $\HH^*(R)/{\widetilde{\mathcal N}}$ is also a finitely generated
  $\Bbbk$-algebra.
\end{prop}

\begin{proof} Let $\Delta_R$ be the diagonal subalgebra of the
  Hochschild cohomology ring. By Theorem \ref{gr-cent}, $\Delta_R$ is
  isomorphic as a graded $\Bbbk$-algebra to the graded center of
  $\mathcal E_R$, so $\Delta_R$ is also finitely generated. By
  Proposition \ref{decomposition}, $\HH^*(R)/{\widetilde{\mathcal N}}$
  is a quotient of $\Delta_R$, so the result follows.
\end{proof}

\noindent We also have the following immediate consequence:

\begin{thm} Let $R$ be a finite dimensional Koszul $\Bbbk$-algebra
  having the property that its Ext-algebra is commutative, and let
  $\widetilde{\mathcal N}$ be the ideal of $\HH^*(R)$ generated by the
  homogeneous nilpotent elements. Then $\HH^*(R)/{\widetilde{\mathcal
      N}}$ is a finitely generated $\Bbbk$-algebra.
\end{thm}

\begin{proof} Assume that $S=\mathcal E_R$ is commutative. If char
  $\Bbbk = 2$, then $Z_{\gr}(S)=S$ so the graded center is finitely
  generated. If the characteristic of $\Bbbk$ is different from 2,
  then every element $x\in Z_{\gr}(S)$ of odd degree is nilpotent with
  nilpotency index 2. So the graded center of $S$ decomposes as
  $Z_{\gr}(S)=S_{\ev}\oplus I$ where $S_{\ev}$ denotes the even degree
  part of $S$, and $I$ is a two sided ideal of the graded center
  generated by nilpotent elements of odd degree. This means that we
  have an isomorphism of graded algebras between $\Delta_R$ modulo the
  ideal generated by its nilpotent elements and $S_{\ev}$ modulo its
  nilpotent elements. Since $S$ is a Koszul algebra, $S_{\ev}$ is
  finitely generated and the result follows.
\end{proof}

\noindent The following example shows that we can have nonzero
elements in $\Delta^n_M$ that are nilpotent.

\begin{Example} We show that there are periodic linear modules $M$
  with constant Betti number 2 such that there is a nonzero nilpotent
  element $\eta\in\Delta^n_M$. We let $$R=\Bbbk\langle
  x,y\rangle/{\langle xy-qyx,x^2,y^2\rangle}$$ where $0\ne
  q\in\Bbbk$. Let $\overline x$ and $\overline y$ denote the residue
  classes in $R$ of $x$ and $y$ respectively.  Assuming $q\ne 0$, $R$
  is a special biserial, selfinjective Koszul algebra.  We now let
\[M=\coker\xymatrix{(R^2\hskip 3pt\ar^{\left(
\begin{smallmatrix}
-\overline y&0\\ \phantom{-}\overline x&q\overline y\end{smallmatrix}\right)}[rr]&&
\hskip 2pt R^2).}\]
We note that $M$ is four-dimensional with basis
$t_1,t_2,b_1,b_2$ such that $t_1\overline x=b_1$,
$t_1\overline y=b_2$, $t_2\overline x=b_2$, and
$t_2\overline y=0$.  We see that $M$ is a string module with
shape
\[\xymatrix{
&t_1\ar@{-}[dl]^{\overline x}\ar@{-}[dr]^{\overline y}&&t_2\ar@{-}[dl]^{\overline x}\\
b_1&&b_2
}\]
\noindent Let $p\colon R^2\to M$ be the map sending
$(1,0)^t$ to $t_1$ and $(0,1)^t$ to $t_2$.  Viewing $M$ as a graded
module generated in degree $0$, a minimal graded projective
resolution of $M$ is given by
\[\xymatrix@1{\cdots\ar[r] & R^2(-2)\ar[rr]^(.5){\left(\begin{smallmatrix}
-\overline y&0\\ \phantom{-}\overline x&q\overline y\end{smallmatrix}\right)}&& R^2(-1)\ar[rr]^(.5){\left(\begin{smallmatrix}
-\overline y&0\\ \phantom{-}\overline x&q\overline y\end{smallmatrix}\right)}&& R^2\ar[r]^{p}& M\ar[r] & 0.
}\]
\noindent Therefore, for each $n$, $\Omega^nM=M(-n)$, so $M$ is a
periodic linear module with constant Betti numbers equal to $2$.

\medskip
\noindent We now define a nonzero nilpotent element
$\eta\in\Delta_M^n$.  Let $N$ be the submodule of $M$ generated by
$t_2$.  We note that $M/N\cong N$.  Let $\eta\in \Delta^1_M$ be given
by composition of the following maps:
\[R^2(-1)	\stackrel{p(-1)}{\longrightarrow} M(-1) \stackrel{\pi}{\longrightarrow} (M/N)(-1)\stackrel{\cong}{\longrightarrow}N(-1){\hookrightarrow}M(-1),
\]
where $\pi$ is the canonical surjection.
The map $\eta\colon R^2(-1)\to M(-1)$ lifts as follows:
\[
\xymatrix{
\cdots \ar[r] &R^2(-2)\ar[d]^{l^1}\ar[rr]^(.50){\left(\begin{smallmatrix}
-\overline y&0\\ \phantom{-}\overline x&q\overline y\end{smallmatrix}\right) }&&R^2(-1)\ar[d]^{l^0}
\ar[dr]^{\eta}
\\
\cdots \ar[r] &R^2(-2)\ar[rr]^(.50){\left(\begin{smallmatrix}
-\overline y&0\\ \phantom{-}\overline x&q\overline y\end{smallmatrix}\right) }&&R^2(-1)\ar^{p(-1)}[r]&M(-1)\ar[r]&0
}
\]
where $l^i$ is given by the matrix
$\left(\begin{smallmatrix}0&0\\(-1/q)^i&0\end{smallmatrix}\right)$.
It is clear that $\eta$ is nonzero and that $\eta^2=0$. \
\end{Example}

\noindent We end this section with the following observation:

\begin{Remark}
  Let $R$ be a Koszul algebra, $M$ a linear $R$-module, and let
  $\eta\colon P_M^n\rightarrow M$ be a map representing an element of
  $\Ext^n_R(M,M)$ for some $n\geqslant 1$ and assume also that the
  image of $\eta$ is not entirely contained in the radical of
  $M$. This can happen for instance if $\eta$ is a surjective
  homomorphism.  We claim that in this case, $\eta$ cannot represent
  the zero element in cohomology. To see this, assume by contradiction
  that $\eta$ factors through $P_M^{n-1}$.  Since the minimal
  resolution of $M$ is linear, $\eta$ must then equal a matrix with
  entries in $M\mathfrak r$, but on the other hand, by
  Proposition~\ref{decomposition}, $\eta=\overline\eta+\varepsilon$ where
  $\overline{\eta}$ is a matrix with entries in $M_0$, and
  $\varepsilon$ is a matrix with entries in $M\mathfrak r$. This
  implies that $\overline{\eta}=0$. This contradicts our assumption,
  and the assertion follows.
\end{Remark}

\section{Liftings}

\noindent In this section we present some technical results about
consecutive liftings associated with elements of $\Delta_M$ and of
$\mathcal N_M$. Let $\eta\colon P_M^n\rightarrow M$ represent an
element of $\Ext_R^n(M,M)$. Consider the following diagram
$$\xymatrix{
\cdots\ar[r]& P_M^{n+m}\phantom{x} \ar[r]^{d_M^{m+n}} \ar[d]^{l^m} & \cdots\ar[r] & P_M^{n+1} \ar[r]^{d_M^{n+1}}
\ar[d]^{l^1} & P_M^n \ar[d]^{l^0}\ar[dr]^{\eta} & \\\cdots\ar[r]&  P_M^m\phantom{x} \ar[r]^{\phantom{xx}d_M^m}
& \cdots\ar[r] & P_M^1 \ar[r]^{d_M^1} & P_M^0 \ar[r]^{d_M^0} &\ M\phantom{xx}
\\}$$
where we view each lifting $l^i$ as a matrix with entries in $R$.

\begin{lemma} Let $R$ be a Koszul algebra, let $M$ be a linear
  $R$-module, and let $(P_M^\bullet, d_M^\bullet)$ be a linear
  resolution of $M$. Let $\eta\colon P_M^n\rightarrow M$ be a $1\times
  k$ matrix with entries in $M$ representing an element of
  $\Ext_R^n(M,M)$, where $k$ is the number of indecomposable summands
  of $P_M^n$.
\begin{enumerate}
\item[(i)] Suppose that $\Im \eta\subseteq M\mathfrak r$. Then the
  lifting $l^0$ can be chosen in such a way so that its entries are
  all in $\mathfrak r$.
\item[(ii)] Suppose that $\eta=[m'_1,\ldots,m'_k]$ where each $m'_i\in
  M_0$. Then the lifting $l^0$ can be chosen in such a way so that its
  entries are all in $R_0$.
\end{enumerate}
\end{lemma}

\begin{proof} Since the image of $\eta$ is in the radical of $M$,
  $\eta$ is a sum of graded homomorphisms of degrees greater than
  $-n$. Without loss of generality we may assume that $\eta$ is a
  graded homomorphism of degree $s>-n$.  As $d_M^0$ is of degree zero,
  the lifting $l^0$ is also of degree $s$ and can be chosen with
  entries in $\mathfrak r$.  The degree $0$ case is similar.
\end{proof}

\noindent It turns out that if one of the liftings has its entries in
$\mathfrak r$, or entirely in $R_0$, then so do all its successors:

\begin{lemma}\label{successive liftings} Keeping the above notation,
  let $R$ be a Koszul algebra and let $M$ be a linear $R$-module. Let
  $\eta\colon P_M^n\rightarrow M$ be a $1\times k$ matrix with entries
  in $M$ representing an element of $\Ext_R^n(M,M)$, where $k$ is the
  number of indecomposable summands of $P_M^n$.
\begin{enumerate}
\item[(1)] Assume that for some $i$, the lifting $l^i$ has all its
  entries in $\mathfrak r$. Then we may choose the subsequent liftings
  $l^{i+1}, l^{i+2},\dots$ in such a way that all their entries are
  also in $\mathfrak r$.
\item[(2)] Assume that for some $i$, the lifting $l^i$ has all its
  entries in $R_0$. Then we may choose the subsequent liftings
  $l^{i+1}, l^{i+2},\dots$ in such a way that all their entries are
  also in $R_0$.
\end{enumerate}
\end{lemma}

\begin{proof} First the radical
case. Consider the following commutative diagram:
$$\xymatrix{ P_M^{n+i+1}
\ar[rr]^-{d_M^{n+i+1}}
\ar[d]^{l^{i+1}} & & P_M^{n+i}\ar[d]^{l^i}\\P_M^{i+1} \ar[rr]^{d_M^n} & &P_M^i}$$

\noindent Since the entries of $l^id_M^{n+i+1}$ are in $\mathfrak r^2$
and those of the differentials are in $R_1$ it follows that we may
choose $l^{i+1}$ as a matrix with entries in $\mathfrak r$. The degree
$0$ case is similar.
\end{proof}

\noindent In regard to the second part of the previous result we
observe that if the lifting $l^i$ has all its entries in $R_0$, then
there exists in fact a unique lifting $l^{i+1}$ having all its entries
in $R_0$ since the difference of any two such liftings would have to
factor through the radical of $P_M^{n+i}$.

\smallskip

\noindent The following corollary discusses two situations in which
all the liftings of a nonzero element of $\Delta_M^n$ must be nonzero
matrices with entries in $R_0$:

\begin{cor}\label{non-zero-liftings} Let $R$ be a Koszul algebra and
  let $M$ be a linear $R$-module. Let $\eta\colon P_M^n\rightarrow M$
  be a map representing a nonzero homogeneous element of degree $n$ in
  $\Delta_M$ for some $n\geqslant 1$. Assume that either $\eta$ is
non-nilpotent, or that $R$ is selfinjective. Then all the successive liftings $l^i$ of $\eta$
are nonzero and can be chosen to be matrices whose nonzero entries are all in $R_0$.
\end{cor}

\begin{proof} Since $\eta$ is nonzero it is clear that the first
  lifting $l^0\ne 0$ and that we may choose all the other liftings
  with entries in $R_0$. Assume first that $\eta$ is non-nilpotent. If
  a lifting $l^i$ was zero, then all the successive liftings would
  also be zero, implying that $\eta$ must be nilpotent. Let us
  consider now the situation when $R$ is selfinjective.  If one of
  these liftings was zero, then the previous one would have to factor
  through its injective envelope hence its entries would not be in
  $R_0$.
\end{proof}

\noindent We continue with two observations. They hold in the more
general context when $R=R_0\oplus R_1\oplus \cdots$ is a graded
$\Bbbk$-algebra with $R_0=\Bbbk\times\cdots\times\Bbbk$. Assume that
$g\colon P\rightarrow Q$ is a degree zero homomorphism between two
projective $R$-modules generated in the same degree. Let $m$ and $n$
be the number of indecomposable summands of $P$ and of $Q$
respectively. Then $g$ can be represented as an $n\times m$ matrix
with entries in $R_0$. We have the following:

\begin{lemma}\label{nice-map-between-projectives}
  Let $\Bbbk$ be a field, let $R=R_0\oplus R_1\oplus \cdots$ be a
  graded algebra with $R_0=\Bbbk\times\cdots\times\Bbbk$, and let $P$
  and $Q$ be two finitely generated projective $R$-modules generated
  in the same degree. Let $g\colon P\rightarrow Q$ be a degree zero
  homomorphism that can be represented by a matrix with entries in
  $R_0$. Then we have $P=\Ker g\oplus P'$ and $Q=\Im g\oplus Q'$,
  $g|_{P'}\colon P'\rightarrow\Im g$ is an isomorphism, and $Q'$ is
  isomorphic to the cokernel of $g$.
\end{lemma}

\begin{proof} It is enough to prove that the cokernel $M$ of $g$ is a
  direct summand of $Q$. Assume not. Without loss of generality we may
  assume that both $P$ and $Q$ are generated in degree zero. So $M$ is
  also generated in degree zero. Let $P^1\to P^0\to M\to 0$ be a
  minimal projective presentation of $M$. It is clear that $P^0$ is
  generated in degree zero and $P^1$ is generated in degree(s) one and
  higher.  We have the following commutative exact diagram of graded
  modules and degree zero homomorphisms:
\[
\xymatrix{
& P\ar[r]^{g}\ar[d]^{h^1}&Q
\ar[r]^{p}\ar[d]^{h^0}&M\ar[r]\ar@{=}[d]&0\\
& P^1\ar[r]^{d^1}&P^0
\ar[r]^{d^0}&M\ar[r]&0
}\]
where $p$ is the canonical surjection.  Since $P^1$ is generated in
degrees higher than $P$, the map $h^1$ is zero, so $h^0$ factors
through $M$. This implies that $d^0$ splits so $M$ is projective.  The
result follows.
\end{proof}

\noindent We mention without proof the following dual result.

\begin{lemma}\label{nice-map-between-injectives}
  Let $\Bbbk$ be a field, let $R=R_0\oplus R_1\oplus \cdots$ be a
  graded algebra with $R_0=\Bbbk\times\cdots\times\Bbbk$, and let $E$
  and $I$ be two finitely generated injective $R$-modules cogenerated
  in the same degree. Let $g\colon E\rightarrow I$ be a degree zero
  homomorphism that can be represented by a matrix with entries in
  $R_0$. Then we have $E=\Ker g\oplus E'$ and $I=\Im g\oplus I'$,
  $g|_{E'}\colon E'\rightarrow\Im g$ is an isomorphism, and $I'$ is
  isomorphic to the cokernel of $g$. \qed \end{lemma}

\noindent We have the following application of \ref{nice-map-between-projectives}:

\begin{prop} Let $R$ be a graded algebra with
  $R_0=\Bbbk\times\cdots\times\Bbbk$ and let $f\colon L\to M$ be a
  degree zero monomorphism between two non-projective modules $L$ and
  $M$ generated in the same degree. Then there exist induced degree
  zero monomorphisms $\Omega f\colon \Omega L\to\Omega M$ and $l\colon
  P_L\to P_M$ where $P_L$ and $P_M$ denote the corresponding
  projective covers.
\end{prop}

\begin{proof} We may assume without loss of generality that both $L$
  and $M$ are generated in degree zero. We have the following
  commutative diagram with exact rows:
\[
\xymatrix{ 0\ar[r]& \Omega L\ar[r]\ar[d]^{\Omega f}&P_L
  \ar[r]^{p}\ar[d]^{l}&L\ar[r]\ar[d]^{f}&0\\
  0\ar[r]& \Omega M\ar[r]&P_M \ar[r]&M\ar[r]&0 }\]
Assume that $i\colon\Ker l\to P_L$ is not zero, and observe that both
$\Omega L$ and $\Omega M$ are generated in positive degrees. By Lemma
\ref{nice-map-between-projectives}, this kernel is isomorphic to a
direct summand of $P_L$, hence $\Hom_R(\Ker l,\Omega L)_0=0$. This
contradicts the fact that the composition $fpi=0$. So $l$ is a
monomorphism and the Snake Lemma implies
that $\Omega f$ is also a monomorphism.
\end{proof}

\noindent As an immediate consequence, we obtain:

\begin{prop}\label{successive monos} Let $R$ be a Koszul algebra and
  let $M$ be an indecomposable linear $R$-module. Let $n\geqslant 1$
  and let $0\ne \eta\colon P_M^n\rightarrow M$ represent an element of
  $\Delta_M^n$ and, for each $i\geqslant 0$, let $l^i\colon
  P_M^{n+i}\rightarrow P_M^i$ be liftings of $\eta$ with entries in
  $R_0$. Assume that $l^j$ is a monomorphism for some $j\geqslant
  0$. Then all the subsequent liftings $l^{j+k}$ with $k\geqslant 0$
  are monomorphisms.
\end{prop}

\begin{proof} Assume that the lifting $l^j$ is a monomorphism. This
  means that the induced morphism
  $f^{j+1}\colon\Omega^{n+j+1}M\to\Omega^{j+1}M$ is also a
  monomorphism so we have the following commutative diagram:
\[
\xymatrix{
0\ar[r]& \Omega^{n+j+2}M\ar[r]\ar[d]^{f^{j+2}}&P^{n+j+1}_M
\ar[r]\ar[d]^{l^{j+1}}&\Omega^{n+j+1}M\ar[r]\ar[d]^{f^{j+1}}&0\\
0\ar[r]& \Omega^{j+2}
M\ar[r]&P^{j+1}_M
\ar[r]&\Omega^{j+1}M\ar[r]&0
}\]
\noindent where the lifting $l^{j+k}$ is a matrix with entries in
$R_0$ by Lemma \ref{successive liftings}. Clearly $l^{j+1}\ne 0$ since
$f^{j+1}$ is a monomorphism. By the previous result, $l^{j+2}$ is also
a monomorphism and so is $f^{j+2}$. The result follows now by
induction. \end{proof}

\noindent It turns out that one can use Lemma
\ref{nice-map-between-projectives} to show that in some interesting
cases we have a similar behavior to the one in Proposition
\ref{successive monos} when one of the liftings is an epimorphism:

\begin{prop}\label{successive--> epis} Let $R$ be a Koszul algebra and
  let $M$ be an indecomposable linear $R$-module of infinite
  projective dimension with the property that none of its syzygies has
  a nonzero projective summand. Assume that for some $n\geqslant 1$,
  $\eta\colon P_M^n\rightarrow M$ represents a nonzero element of
  $\Delta_M^n$. For each $i\geqslant 0$, let $l^i\colon
  P_M^{n+i}\rightarrow P_M^i$ be liftings of $\eta$ with entries in
  $R_0$. Assume that $l^j$ is an epimorphism for some $j\geqslant
  0$. Then all the previous liftings $l^{j-k}$ with $0\leqslant k\leqslant j$ are
  epimorphisms.
\end{prop}

\begin{proof} Assume that the lifting $l^{j-1}$ is not onto. Then, by
  \ref{nice-map-between-projectives}, its cokernel $Q$ is projective
  and generated in degree $j-1$. We have the following commutative
  diagram:
\[
\xymatrix{ & P^{j+n}_M\ar[r]\ar[d]^{l^{j}}&P^{j+n-1}_M
  \ar[r]\ar[d]^{l^{j-1}}&\Omega^{j+n-1}M\ar[r]\ar[d]^{f^{j-1}}&0\\
  & P^j_M\ar[r]^{d_M^j}&P^{j-1}_M
  \ar[r]\ar[d]^{\pi}&\Omega^{j-1}M\ar[r]&0\\
  &&Q&& }\]
Since $l^j$ is onto, we have that $\pi d_M^j=0$, so $\pi$ must factor
through the syzygy $\Omega^{j-1}M$. This induces an epimorphism
$\Omega^{j-1}M\to Q$ and we get a contradiction to our assumption.
\end{proof}

\begin{cor} Let $R$ be a selfinjective Koszul algebra and let $M$ be
  an indecomposable non projective linear $R$-module. Assume that for
  some $n\geqslant 1$, $\eta\colon P_M^n\rightarrow M$ represents a
  nonzero element of $\Delta_M^n$. For each $i\geqslant 0$, let
  $l^i\colon P_M^{n+i}\rightarrow P_M^i$ be liftings of $\eta$ with
  entries in $R_0$. Assume that $l^j$ is an epimorphism for some
  $j\geqslant 0$. Then all the previous liftings $l^{j-k}$ with $0\leqslant
  k\leqslant j$ are epimorphisms. \qed
\end{cor}

\section{Modules with bounded Betti numbers}

In this section we study periodic modules, and more generally, modules
with bounded Betti numbers, that is, having complexity one. The set-up
is more general than Koszul in the finite dimensional selfinjective
case. We start with the following:

\begin{prop} Let $R$ be a selfinjective algebra and let $M$ be an
  indecomposable module. The following statements are equivalent.
\begin{enumerate}
\item[(a)] The module $M$ is periodic.
\item[(b)] The module $M$ has complexity one and there exists a
  non-nilpotent element in $\Ext^*_R(M,M)$ of positive degree.
\end{enumerate}
\end{prop}
\begin{proof} Assume that the module $M$ is periodic of period $n$ for
  some $n\geqslant1$, so $M\cong\Omega^nM$.  Then $M$ has
  complexity one and the extension
\[0\to M\to P_M^{n-1}\to \cdots \to P_M^1\to P_M^0\to M\to
0\]
induced from the minimal projective resolution of $M$ gives rise to a
non-nilpotent element in $\Ext^*_R(M,M)$ of positive degree.  For the reverse
implication, assume that the module $M$ has complexity one and that
there is a non-nilpotent element $\eta$ of positive degree $n$ in
$\Ext^*_R(M,M)$.  Since $\Ext^{k}_R(M,M)$ can be identified with
$\underline\Hom_R(\Omega^kM,M)$ for each $k\geqslant 1$, we may consider
$\eta$ as being a homomorphism that we also denote by $\eta\colon
\Omega^nM\to M$.  Then a power $\eta^t$ of $\eta$ is represented by
the composition
\[f_t=\eta\Omega^n(\eta)\cdots
\Omega^{(t-1)n}(\eta),\]
\sloppy where $\Omega^{jn}(\eta)$ are liftings of $\eta$.  Since
$\eta^t$ is nonzero for all $t\geqslant 1$, the homomorphism $f_t$ is
nonzero for all $t\geqslant 1$.  Since all the modules $\Omega^{jn}M$
are indecomposable, and their lengths are bounded, we infer by the
Harada-Sai Lemma that one of the homomorphisms $\Omega^{jn}_R(\eta)$
is an isomorphism.  That is, $\Omega^{(j+1)n}M$ is isomorphic to
$\Omega^{jn}M$ for some $j$ with $1\leqslant j \leqslant t - 1$.
Since $R$ is selfinjective, it follows that $M$ is periodic of period
dividing $n$.
\end{proof}

\noindent Since tensoring a minimal projective $R^e$-resolution of $R$
with $R/\mathfrak r$ yields a minimal projective $R$-resolution of
$R/\mathfrak r$, it follows that, viewed as an $R^e$-module, $R$ has
complexity one if and only if all the simple $R$-modules have
complexity one.  The following is a direct consequence of our
discussion.

\begin{cor}\label{cor:periodic}
Let $R$ be a selfinjective algebra. Then the following statements are
equivalent.
\begin{enumerate}
\item[(a)] $R$ is a periodic algebra.
\item[(b)] As an $R^e$-module, $R$ has complexity one, and there
  exists a non-nilpotent element of positive degree in the Hochschild
  cohomology ring of $R$.
\item[(c)] All the simple $R$-modules have complexity one, and there
  exists a non-nilpotent element of positive degree in the Hochschild
  cohomology ring of $R$.
\end{enumerate}
\end{cor}

\begin{Remark} It would be interesting to know whether a selfinjective
  algebra $R$ with the property that it has complexity one over its
  enveloping algebra, is necessarily itself a periodic algebra.
  Specifically, is the existence of a non-nilpotent element of
  positive degree in the Hochschild cohomology simply a consequence?
  Let us assume for a moment that the ground field is algebraically
  closed. Using geometric methods, Dugas has proved in \cite{D1} that
  if a simple module over a selfinjective algebra has complexity one,
  then that simple module must be periodic. Hence all the simple
  modules having complexity one over $R$, are periodic. In the finite
  representation case, he also proved in \cite{D2} that this is
  equivalent to the algebra being periodic.  Not much is known in the
  infinite representation case. Note that we also have the following
  result: ``If all the simple modules are periodic, then the algebra
  is selfinjective" by \cite{GSS}.\end{Remark}

\noindent We recall from \cite{GSS} that if every simple $R$-module is
periodic, and if the base field is algebraically closed, then there
exists an automorphism $\varphi$ of $R$ such that the $R^e$-modules
$\Omega_{R^e}^nR$ and $_1 R_\varphi$ are isomorphic for some integer
$n\geqslant 1$. We have the following:

\begin{prop}
  Let $R$ be a selfinjective algebra over an algebraically closed
  field and assume that all the simple modules are periodic.  Let
  $\varphi$ be an automorphism of $R$ such that the $R^e$-modules
  $\Omega_{R^e}^nR$ and $_1 R_\varphi$ are isomorphic for some integer
  $n\geqslant 1$. The following statements are equivalent.
\begin{enumerate}
\item[(a)] The Hochschild cohomology ring of $R$ has a non-nilpotent
  element of positive degree.
\item[(b)] $R$ is a periodic algebra.
\item[(c)] The $R$-bimodules $R$ and $_1 R_{\varphi^t}$ are isomorphic
  for some integer $t\geqslant 1$.
\item[(d)] The automorphism $\varphi^t$ is inner for some integer
  $t\geqslant 1$.
\end{enumerate}
\end{prop}

\begin{proof}
(a) implies (b): This follows from Corollary~\ref{cor:periodic} since all the simple modules have complexity one. (b) implies (c):  This follows immediately from the fact that
$\Omega_{R^e}^nR\cong {_1 R_{\varphi}}$. (c) implies (a): Straightforward.

(c) equivalent with (d):  Well-known, see \cite[Theorem 55.11]{CR}.  
\end{proof}

\noindent  As a corollary of \cite[Theorem 1.4]{GSS}  and
  the remark following that result, we also have the following:
\begin{prop}
If $R=\Bbbk\mathcal Q/I$ is a finite dimensional algebra over a finite field and all
simple $R$-modules are periodic, then $R$ is a periodic algebra. \qed
\end{prop}

\noindent Now we show there are situations where periodic modules give
rise to non-nilpotent elements that are actually in the {\it graded
  center} of their $\Ext$-algebra, and consequently, by \ref{gr-cent},
give rise to non-nilpotent elements in the Hochschild cohomology ring
of the algebra itself. Recall that an algebra $R$ satisfies the
\textup{\textbf{Fg}} condition (see \cite{EHSST} for instance), if its
Ext-algebra $\mathcal E_R$ is a finitely generated $H$-module for some
commutative Noetherian graded subalgebra $H\subseteq \HH^*(R)$ with
$H^0=\HH^0(R)$.  There is a homomorphism of graded rings
$\gamma_M\colon \HH^*(R)\to \Ext^*_R(M,M)$ so $\Ext^*_R(M,M)$ becomes
an $H$-module via the action of $\gamma_M$. If $M$ is a periodic
$R$-module then its Ext-algebra is nonzero, so by \cite[Proposition
2.1]{SS}, its annihilator in $H$ is a proper ideal. By
\cite[Proposition 2.1]{EHSST}, the variety of $M$, $V_H(M)$ has Krull
dimension one, since $M$ being periodic, has complexity one. So there
are nonzero homogeneous elements in $H$ of arbitrary large degrees.

\begin{prop}
  Let $R$ be a selfinjective algebra over an algebraically closed
  field satisfying \textup{\textbf{Fg}}.  If $M$ is a periodic module,
  then there exists a non-nilpotent element of positive degree in the graded
  center of the $\Ext$-algebra of $M$ and also, there
  exists a non-nilpotent element of positive degree
  in the Hochschild cohomology ring of $R$.
\end{prop}

\begin{proof}
  Assume that $M$ is a periodic module.  By the above remarks the
  dimension of the variety of $M$ is one, and there are nonzero
  homogeneous elements in $H$ of arbitrary large degrees.  Let $\eta$
  be a homogeneous element of degree $n\geqslant 1$ in the Hochschild
  cohomology ring $\HH^*(R)$ represented as a homomorphism
  $\eta\colon \Omega^n_{R^e}R\to R$. We can form the following
  pushout,
\[
\xymatrix{
0\ar[r] & \Omega^n_{R^e}R\ar[r]\ar[d]^\eta & P^{n-1} \ar[r]\ar[d] &
\Omega^{n-1}_{R^e}R\ar[r]\ar@{=}[d] & 0\\
0\ar[r] & R\ar[r] & M_\eta\ar[r] & \Omega^{n-1}_{R^e}R\ar[r] & 0
}\]
where
\[0\to \Omega^n_{R^e}R\to P^{n-1}\to \cdots\to P^1\to P^0\to R\to
0\]
is the start of a minimal projective resolution of $R$ over $R^e$.  By
\cite[Proposition 4.3]{EHSST}, the variety of the module
$M\otimes_RM_\eta$ is given by the ideal
defining the variety of $M$ and the element $\eta$.  As in the proof
of \cite[Theorem 5.3]{EHSST}, we may choose $\eta$ such that the
variety of $M\otimes_R M_\eta$ is trivial.  Hence $M\otimes_R M_\eta$ is a
projective $R$-module, again by \cite[Theorem 2.5]{EHSST}.

\smallskip

\noindent The image of $\eta$ under the homomorphism $\gamma_M\colon \HH^*(R)\to
\Ext^*_R(M,M)$ is given by
\[\gamma_M(\eta)\colon 0 \to M\to M\otimes_R M_\eta\to M\otimes_R P^{n-2}\to
\cdots \to M\otimes_R P^0\to M\to 0\]
and $\gamma_M(\eta)$ is contained in the graded center of
$\Ext^*_R(M,M)$ by \cite[Corollary 1.3]{SS}. Since $M\otimes_RM_\eta$
is projective, the extension $\gamma_M(\eta)$ is given by the first
$n$ terms in a minimal projective resolution of $M$ (and the period
of $M$ is a divisor of $n$). Therefore $\gamma_M(\eta)$ is a
non-nilpotent element in the graded center of $\Ext^*_R(M,M)$.
Consequently $\eta\in \HH^*(R)$ is also a non-nilpotent element. This
completes the proof.
\end{proof}

\noindent We have the following immediate consequence of our
discussion so far:

\begin{prop} Let $R$ be a selfinjective algebra over an algebraically
  closed field. Assume that $R$ satisfies the \textup{\textbf{Fg}} condition. Then
  $R$ is periodic if and only if all the simple $R$-modules are
  periodic. \qed
\end{prop}

\noindent We give now an example showing that the existence of a
periodic module $M$ over a selfinjective (even Koszul) algebra $R$
does not necessarily imply the existence of a non-nilpotent element of
positive degree in the graded center of its $\Ext$-algebra. Obviously
the \textup{\textbf{Fg}} condition is not satisfied in this case, and the graded
center of the $\Ext$-algebra of $M$ is spanned as a vector space by the
identity map $1_M$.

\begin{Example}
Let
\[R = \Bbbk(\xymatrix{ 1\ar@(dl,ul)^x\ar@(ur,dr)^y})/\langle x^2, xy +
qyx, y^2\rangle\]
be the quantum plane, where $\Bbbk$ is a field and $q$ is nonzero and not
a root of unity in $\Bbbk$.  Let $M$ be the $R$-module given by the
representation
\[
\xymatrix{ \Bbbk^2\ar@(dl,ul)^{\left(\begin{smallmatrix}0 & 1\\ 0 &
        0\end{smallmatrix}\right)}
\ar@(ur,dr)^{\left(\begin{smallmatrix}0 & 0\\ 0 & 0\end{smallmatrix}\right)}}
\]
Then $R$ is a selfinjective algebra of dimension four, and $M$ is an
indecomposable periodic $R$-module of period one.  As a $\Bbbk$-vector
space, the endomorphism ring of $M$, is spanned by the set $\{ 1_M,
f=\left(\begin{smallmatrix} 0 & 1\\ 0 & 0\end{smallmatrix}\right)\}$.
The homomorphism
\[\mu=\left(\begin{smallmatrix} 1 & 0\\ 0 &
    -q\end{smallmatrix}\right)\]
is an isomorphism $\mu\colon \Omega^1_RM\to M$. Consider the
pushout
\[\xymatrix{
0\ar[r] & \Omega_R^1M \ar[r]^-\iota\ar[d]^\mu & P^0\ar[r]^\pi\ar@{=}[d] &
M\ar[r]\ar@{=}[d] & 0\\
0\ar[r] & M \ar[r]^-{\iota\mu^{-1}} & P^0\ar[r]^\pi & M\ar[r] & 0}
\]
\sloppy and denote the lower exact sequence again by $\mu$. Then, this
sequence together with the set $\{1_M,f\}$ generates $\Ext^*_R(M,M)$
as an algebra over $\Bbbk$.  The relations for these generators are
$f^2 = 0$ and $\mu f + q f\mu = 0$.  One can then show that the graded
center of $\Ext^*_R(M,M)$ is just the one dimensional $\Bbbk$-space
spanned by $1_M$.
\end{Example}

\noindent Our next aim is to describe another setting where having
non-nilpotent elements in the $\Ext$-algebra of a simple module
implies the existence of non-nilpotent elements of positive degree in the Hochschild
cohomology ring of the algebra.  We start with the following
preparatory result.

\begin{prop}
  Let $S$ be a graded algebra such that $S$ is a finitely generated
  module over $Z_{\gr}(S)$.  Assume that $S_{\geqslant 1}$ has a
  non-nilpotent element.  Then there exists a non-nilpotent element in
  $Z_{\gr}(S)_{\geqslant 1}$.
\end{prop}

\begin{proof}
Let $\{f_1,\ldots,f_n\}$ be a finite set of homogeneous generators for
$S$ as a module over $Z_{\gr}(S)$.  Let $m=\max\{\deg(f_i)\}_{i=1}^n$.  Let
$z$ be a non-nilpotent element in $S_{\geqslant 1}$.  Choose
$t$ such that $\deg(z^t) > m$. Then
\[z^t = f_1x_1 + \cdots + f_nx_n\]
for some $x_i$ in $Z_{\gr}(S)$ with all the nonzero $x_i$ in
$Z_{\gr}(S)_{\geqslant 1}$.  Then for each $j\geqslant 1$, we have
$z^{tj}\neq 0$ and it can be written as
\[
z^{tj} = \sum_{i_1,i_2,\ldots,i_n}
w_{i_1,i_2,\ldots,i_n}(f_1,\ldots,f_n) x_1^{i_1}x_2^{i_2}\cdots x_n^{i_n}
\]
where $w_{i_1,i_2,\ldots,i_n}(f_1,\ldots,f_n)$ are linear combinations
of words over $\mathbb{Z}$ in the set $\{ f_1,\ldots,f_n\}$.  Since
$j$ can be arbitrary large, it follows that there exists some $k$ such that
$x_k$ is non-nilpotent in $Z_{\gr}(S)_{\geqslant 1}$.  This completes
the proof.
\end{proof}

\noindent Letting $S=\mathcal E_R$ be the Ext-algebra of $R$, the next
result follows immediately, using the fact that
there is a surjective homomorphism from the Hochschild cohomology ring
of a Koszul algebra $R$ onto the graded center of the Koszul dual
$\mathcal E_R$.

\begin{cor}
  Let $R$ be a Koszul algebra such that $\mathcal E_R$ is a finitely
  generated $Z_{\gr}(\mathcal E_R)$-module. If there exists a
  non-nilpotent element in $(\mathcal E_R)_{\geqslant 1}$, then exists a
  non-nilpotent element of positive degree in the Hochschild
  cohomology ring of $R$. \qed
\end{cor}

\noindent We return now to the case where $R$ is a Koszul algebra. Let
$M$ be an indecomposable linear module over $R$. We have seen in Proposition~\ref{successive monos}
that if $n\geqslant 1$, and if $\eta\colon P_M^n\rightarrow M$
represents an element of $\Delta_M^n$ with the property that one of
its liftings is a monomorphism, then all subsequent liftings are also
monomorphisms. We have the following:

\begin{thm}\label{l-mono} Let $R$ be a Koszul algebra and let $M$ be
  an indecomposable linear $R$-module. Let $n\geqslant1$ and let
  $\eta\colon P_M^n\rightarrow M$ represent an element of
  $\Delta_M^n$. Assume that for some $j$, the lifting $l^j$ of $\eta$
  having entries in $R_0$ is a monomorphism. Then there exists a
  positive integer $n_0$ such that there is an isomorphism
  $P_M^{n_0n+j}\to P_M^{kn+j}$ of degree $-kn+n_0n$ for each
  $k\geqslant n_0$. Consequently, $M$ is eventually periodic.  In
  particular, if $R$ is also selfinjective, then $M$ is periodic of
  period dividing $n$.
\end{thm}

\begin{proof} Using Proposition~\ref{successive monos} and the notation therein,
each lifting $l^m\colon P_M^{n+m} \to P_M^{m}$ for
  $m\geqslant j$ is a monomorphism of degree $-n$, and therefore each
  induced map $f^{m+1}\colon\Omega^{n+m+1}M\to\Omega^{m+1}M$ is also a
  monomorphism of degree $-n$ for $m\geqslant j$. We have for each
  $k\geqslant 1$, chains of monomorphisms of degree $-n$ given by our
  liftings: $$P_M^{kn+j}\hookrightarrow
  P_M^{(k-1)n+j}\hookrightarrow\cdots\hookrightarrow
  P_M^{n+j}\hookrightarrow P_M^j,$$ \noindent
  and $$\Omega^{kn+j}M\hookrightarrow\Omega^{(k-1)n+j}M\hookrightarrow\cdots\hookrightarrow
  \Omega^{n+j}M\hookrightarrow\Omega^jM.$$

\noindent In the finite dimensional case, the result is now obvious,
  so we prove the general case. To show eventual periodicity, we only
  need to look at the number of indecomposable summands of the
  projective modules involved to see that there exists a positive
  integer $n_0$ such we have an isomorphism
  $P_M^{n_0n+j}((k-n_0)n)\cong P_M^{kn+j}$ for each $k\geqslant
  n_0$. But $l^{n_0n+j}$ being an isomorphism of degree $-n$ implies
  that $f^{s-n}\colon\Omega^sM\to\Omega^{s-n}M$ is a monomorphism for
  $s \geqslant n_0n+j$.  By applying the Snake Lemma to
\[
\xymatrix{
0\ar[r]& \Omega^{(n_0+1)n+j+1}M\ar[r]\ar[d]&P^{(n_0+1)n+j}_M
\ar[r]\ar[d]^{l^{n_0n+j}}&\Omega^{(n_0+1)n+j}M\ar[r]\ar[d]&0\\
0\ar[r]& \Omega^{n_0n+j+1}
M\ar[r]&P^{n_0n+j}_M
\ar[r]&\Omega^{n_0n+j}M\ar[r]&0
}\]
we see that $\Omega^{(n_0+1)n+j}M\to \Omega^{n_0n+j}M$ is both a
monomorphism and an epimorphism.  It follows that $M$ is eventually
periodic.  If $R$ is selfinjective, $M$ being eventually periodic
implies that there is a degree $-n$ isomorphism $\Omega^nM\to M$ so
$M$ is in fact periodic.
\end{proof}

\noindent Let us also recall the following well-known result.

\begin{lemma}\label{zero-iso} Let $R=R_0\oplus R_1\oplus R_2\oplus\cdots$ be a positively
  $\mathbb Z$-graded algebra with $R_0=\Bbbk\times\cdots\times \Bbbk$
  and let $M$ and $N$ be finitely generated graded modules generated
  in degree $0$.  Suppose $f\colon M\to N$ is an $R$-isomorphism and
  write $f=\sum_if_i$, where $f_i\colon M\rightarrow N$ is a degree
  $i$ homomorphism from $M$ to $N$. Then $(f_0)_{|M_0}\colon M_0\to
  N_0$ is an isomorphism.
\end{lemma}

\noindent Let $R$ be a Koszul algebra and let $M$ be a linear
module. Let $(P^{\bullet}_M,d_M^{\bullet})$ be a linear resolution of
$M$, and let $n\geqslant 1$. Denote by $\pi_n$ the projection of
$P_M^n$ onto $\Omega^nM$ induced by the differential $d_M^n$. It is
clear that having a map $\eta\colon P_M^n\to M$ such that $\eta
d_M^{n+1}=0$ is equivalent to having a map
$\overline{\eta}\colon\Omega^nM\to M$ such that
$\overline{\eta}\pi_n=\eta$. So we may think of $\eta$ and
$\overline{\eta}$ as representing the same element in
$\Ext_R^n(M,M)$. We can summarize our discussion:

\begin{thm}\label{summary} Let $R$ be a Koszul algebra and let $M$ be an indecomposable linear $R$-module
of infinite projective dimension.
Consider the following statements:
\begin{enumerate}
\item The module $M$ is periodic.
\item There exist $n\geqslant 1$ and $\eta\in\Delta^n_M$ such that
some lifting of $\eta$ is an isomorphism.
\item There exist $n\geqslant 1$ and $\eta\in\Delta^n_M$ such that
some lifting of $\eta$ is a monomorphism.
\item The module $M$ is eventually periodic.
\end{enumerate}
Then
\[
\xymatrix{( 1) \ar@{=>}[r]& (2) \ar@{<=>}[r] &(3)\ar@{=>}[r]&(4)
}\]
Moreover, if $R$ is selfinjective, then $(1)-(4)$ are all equivalent.
\end{thm}

\begin{proof} To see that (1) implies (2), assume that $f\colon
  \Omega^nM\to M$ is an isomorphism of $R$-modules for some
  $n\geqslant 1$.  Shifting and applying Lemma \ref{zero-iso}, we see
  that $(f_0)_{|(\Omega^nM)_n}\colon(\Omega^nM)_n\to M_n$ is an
  isomorphism of degree $-n$.  Now $f_0$ represents an element
  $\eta\in\Delta^n_M$.  Since the restriction
  $(f_0)_{|(\Omega^nM)_n}\colon( \Omega^nM)_n\to M_n$ is an
  isomorphism, the matrix of the first lifting, $l^0\colon P_M^n\to
  P_M^0$ is an isomorphism.

\smallskip
\noindent It is clear that (2) implies (3).  On the other hand, the
proof of Theorem \ref{l-mono} shows that (3) implies (2). Theorem
\ref{l-mono} also shows that (3) implies (4).

\smallskip
\noindent Finally, if $R$ is selfinjective, (4) implies (1) and the
proof is complete.
\end{proof}

\noindent Let $M$ be a linear module and let $0\ne\eta\in\Delta^n_M$
where $n\geqslant 1$. We have seen in the previous section that we
have a commutative diagram
$$\xymatrix{
\cdots\ar[r]& P_M^{n+m} \ar[r]^{d_M^{m+n}} \ar[d]^{l^m} & \cdots\ar[r] & P_M^{n+1} \ar[r]^{d_M^{n+1}}
\ar[d]^{l^1} & P_M^n \ar[d]^{l^0}\ar[dr]^{\eta} & \\\cdots\ar[r]&  P_M^m \ar[r]^{\phantom{xx}d_M^m}
& \cdots\ar[r] & P_M^1 \ar[r]^{d_M^1} & P_M^0 \ar[r]^{d_M^0} & M
\\}$$
where each lifting $l^i$ is a matrix with entries in $R_0$.
Assume that for some $i$, the lifting $l^i$ is a degree~$-n$ isomorphism of projective $R$-modules. Then from the commutative diagram
$$\xymatrix{ P_M^{n+i+1}
\ar[rr]^-{d_M^{n+i+1}}
\ar[d]^{l^{i+1}} & & P_M^{n+i}\ar[d]^{l^i}\\P_M^{i+1} \ar[rr]^{d_M^n} & &P_M^i}$$
\noindent we get as an immediate application of the Snake Lemma an embedding $\Omega^{n+i+1}M\hookrightarrow\Omega^{i+1}M$ and an onto map $\Omega^{n+i}M\rightarrow\Omega^iM$.

\begin{cor}\label{cx1}Let $R$ be a Koszul algebra and let $M$ be an
  indecomposable linear $R$-module whose Betti numbers are all equal
  to 1. Then $M$ is eventually periodic if and only if $\Delta^n_M\ne
  0$ for some $n>0$.
\end{cor}

\begin{proof} Assume first that there exists a nonzero element $\eta$
  of degree $n>0$ in $\Delta_M$. From the preceding results, each
  lifting $l^i$ of $\eta$ is a matrix with entries in $R_0$, and the
  Betti numbers of $M$ being equal to 1 imply that $l^0$ is an
  isomorphism. Therefore $M$ is eventually periodic by
  Theorem~\ref{summary}.

\smallskip
\noindent For the other direction, assume that $M$ is eventually
periodic of period $n>0$, so there is an integer $n_0$ such that there
is a degree $-n$ isomorphism $\Omega^{n_0+n}M\to\Omega^{n_0}M$. By
looking at $\Omega^{n_0}M$ instead of $M$ we may assume without loss
of generality that $M$ is itself periodic of period $n$. So we have a
commutative diagram
\[\xymatrix{ P_M^{n}
\ar[rr]\ar[d]^{l}&& \Omega^nM\ar[d]^{f}\\P_M^0 \ar[rr]^{d^0_M} &&M}\]
\noindent where $f$ and $l$ are degree $-n$ isomorphisms. In fact, $l$ can be viewed as a matrix with entries in $R_0$.  Letting $\eta=d^0_Ml$ and $l^0=l$ we see that $\eta\in\Delta^n_M$ and is nonzero.
\end{proof}

\begin{Remark} Note that by our proof of the `only if' direction of
  Corollary \ref{cx1}, the following more general result is true: ``If
  $R$ is a Koszul algebra, $M$ is an indecomposable linear $R$-module
  and $M$ is eventually periodic, then $\Delta^n_M\ne 0$ for some
  $n>0$."
\end{Remark}

\noindent The following result requires an element of $\Delta_M^n$ all
of whose liftings are nonzero. Recall that if $R$ is a selfinjective
Koszul algebra, then, by Corollary \ref{non-zero-liftings} if for some
$n\geqslant 1$ there exists a nonzero $\eta\in\Delta_M^n$, then all
the liftings $l^i$ of $\eta$ are nonzero.

\begin{prop}\label{betti1} Let $R$ be a Koszul algebra and let $M$ be
  an indecomposable linear
$R$-module such that the $i$-th Betti number of $M$ is $1$, for some $i\geqslant 1$.
If, for some $n<i$, $\Delta_M^n$ contains an element all of whose liftings
are nonzero, then $M$ is eventually periodic.
If, in addition, $R$ is selfinjective and $\Delta_M^n\ne 0$ for some
$n<i$, then $M$ is periodic of period dividing $n$.
\end{prop}

\begin{proof} Suppose that $\eta$ is a nonzero element of $\Delta_M^n$
  and $n<i$.  Let $(P_M^{\bullet}, d^{\bullet}_M)$ be a linear
  resolution of $M$ and let $l^j$ be liftings of $\eta$ as above. Let
  $s=i-n$.  Then $l^s\colon P_M^{s+n} \to P_M^{s}$ and
  $P_M^{s+n}=P_M^i$ has rank 1.  Since the liftings of $\eta$ are
  assumed to be nonzero, $l^s\colon P_M^i\to P_M^s$ is a monomorphism.
  The result now follows from Theorem \ref{l-mono}.
\end{proof}

\noindent We have the following example:

\begin{Example} First assume that $M$ is a linear module over a
  connected Koszul algebra such that the Betti numbers of $M$ are
  eventually $1$. Let $n_0$ be such that the rank of each $P_M^k$ is 1
  for each $k\geqslant n_0$. Let $n\geqslant n_0$ and let
  $0\ne\eta\in\Delta^n_M$. We may think of $\eta$ as being a degree
  $-n$ map from $\Omega^nM$ to $M$ that does not factor through a
  projective module, and which is given by a matrix whose entries are
  all in $\Bbbk$. But each syzygy $\Omega^nM$ of $M$ is cyclic so
  $\eta$ has to be onto as it maps $(\Omega^nM)_n$ onto $M_n$. In
  other words, $\Delta^n_M\ne 0$ if and only if there is an onto map
  $\Omega^nM\rightarrow M$.

\medskip
\noindent Let us apply this to the following specific situation.  Let
$R=\Bbbk\langle x,y\rangle/{\langle xy-qyx,x^2,y^2\rangle}$ where
$0\ne q\in\Bbbk$. Let $\overline x$ and $\overline y$ denote the
residue classes in $R$ of $x$ and $y$ respectively, and let $M$ be the
cokernel of the multiplication by $\overline x+\overline y$ viewed as
a map $R\rightarrow R$. A linear resolution of $M$ is given by the
following:
$$ \  \cdots \extto{}
R(-n)\extto{\overline x+(-q)^n\overline y} R(-n+1)\extto{\overline
  x+(-q)^{n-1}\overline y} \cdots \extto{} R(-1)\extto{\overline
  x+\overline y} R \to M\to  0.$$
\noindent This means that for each $n\geqslant 1$,
$\Omega^nM\cong\coker(\overline x+(-q)^n\overline y)$. If $q$ is not a
root of unity, then $\Omega^nM\ncong\Omega^mM$ for all $m\ne n$, and
they are all two-dimensional so $\Delta^n_M=0$ for all $n>0$. Hence
$\Delta_M=\Bbbk$ in this case. If $(-q)^n=1$, then $\Omega^nM\cong M$
and it is easy to see that $\Delta_M$ is infinite dimensional, but
finitely generated as a $\Bbbk$-algebra by its degree $0\leqslant k\leqslant n$
parts.
\end{Example}

\noindent Let $R = R_0 \oplus R_1 \oplus \cdots $ be a Koszul
$\Bbbk$-algebra where $R_0 = \Bbbk \times \cdots \times \Bbbk$.  Let
$M$ and $N$ be linear $R$-modules. In an analogous way to the diagonal
subalgebra which was introduced in Section~\ref{sec:two}, we define
the {\it diagonal module} of $\Ext^*_R(M,N)$ as
\[\Delta(M,N)=\bigoplus_{n\geqslant 0}\Ext^n_R(M,N)_{-n}.\]
For each nonnegative integer $n$ we also define
$\Delta^n(M,N)=\Ext^n_R(M,N)_{-n}$.
 Using this definition and the results on liftings from
  Section 3, we show that having some Betti numbers equal to $1$
  forces further occurrences of $1$ as a Betti numbers, in some
  cases.

\begin{prop}\label{prop:comparison_betti_numbers}
Let $R$ be a Koszul algebra.  Let $M$ and $N$ be two linear $R$-modules
with the following properties:
\begin{enumerate}
\item[(i)] no syzygy of $N$ has a nonzero projective direct summand,
\item[(ii)] there exists an element $\eta$ in $\Delta^i(M,N)$ for some
  $i \geqslant 0$ and there exists an element $\theta$ in
  $\Ext^n_R(N,L)$ for some $n \geqslant 1$ and some linear $R$-module
  $L$, such that the Yoneda product $\theta\eta \neq 0$,
\item[(iii)] $\beta_m(M)=\beta_n(N)=1$ for some $m$ with $i \leqslant m < n+i$.
\end{enumerate}
Then $\beta_{m-i}(N)=1$.
\end{prop}

\begin{proof}
Let
\[P^\bullet_M\colon \cdots\to P^1_M \to P^0_M\to M\to 0\]
and
\[P^\bullet_N\colon \cdots\to P^1_N \to P^0_N \to N\to 0\]
be linear projective resolutions of $M$ and $N$, respectively.
Suppose $\eta$ in $\Delta^i(M,N)$ is represented by $\eta \colon
P^i_M\to N$, and $\theta$ in $\Ext^n_R(N,L)$ is represented by
$\theta\colon P^n_N\to L$, where we assume that the Yoneda product
$\theta\eta$ is nonzero.  The element $\theta\eta$ can be represented
by the composition $P^{n+i}_M\xrightarrow{\ell^n(\eta)}
P^n_N\xrightarrow{\theta} L$.  Since the product $\theta\eta$ is
nonzero, the above composition of maps is nonzero, and in
particular, $\ell^n(\eta)$ is nonzero.  Since $\eta$ is in
$\Delta^i(M,N)$, all the liftings $\ell^j(\eta)$ are given by matrices
in $R_0$.  By assumption $\beta_n(N) = 1$, so we infer that
$\ell^n(\eta)$ is surjective.  By Proposition \ref{successive--> epis}
all the previous liftings $\ell^j(\eta)$ for $j=n-1,n-2,\ldots,1,0$
are also surjective. In particular $\ell^{m-i}(\eta)\colon P^m_M\to
P^{m-i}_N$ is surjective where $m-i = 0, 1, \ldots , n-1$.
Since $\beta_m(M)=1$, it follows that
$\beta_{m-i}(N)=1$.
\end{proof}
\noindent We have the following result:

\begin{thm}
Let $R = R_0 \oplus R_1 \oplus \cdots$ be an indecomposable Koszul algebra.  Assume that
 all the simple modules have infinite projective dimension and that no syzygy of a simple module has a nonzero projective
  direct summand.
Assume also, that for each simple
  module, $1$ occurs at least twice as a Betti number.
Then every simple module has a syzygy which is simple and periodic.
\end{thm}

\begin{proof} (a) Fix a simple module $S$.  By assumption there exists
  some $n\geqslant 1$ such that $\beta_n(S)=1$. Let $T$ be a simple
  $R$-module such that $\Ext^n_R(S,T)\neq (0)$. Let
\[P^{\bullet}_S\colon \cdots\to P^1_S\to P^0_S\to S\to 0\]
and
\[P^{\bullet}_T\colon \cdots\to P^1_T\to P^0_T\to T\to 0\]
be linear projective resolutions of $S$ and $T$, respectively.  Let
$\eta$ be a nonzero element in $\Ext^n_R(S,T)$.  Note that
$\eta\in\Delta^n(S,T)$ since $T$ is a simple module. We can represent
$\eta$ as a homomorphism denoted again by $\eta\colon P^n_S\to T$, and
let $l^i\colon P^{n+i}_S\to P^i_T$ be the liftings of $\eta$ as
shown in the following diagram:
\[\xymatrix{
\cdots \ar[r]& P^{n+i}_S\ar[r]\ar[d]^{l^i} & \cdots \ar[r]& P^{n+1}_S\ar[r]\ar[d]^{l^1} & P^n_S\ar[dr]^\eta\ar[d]^{l^0} & \\
\cdots \ar[r]& P^i_T\ar[r] & \cdots \ar[r]& P^1_T\ar[r] & P^0_T\ar[r] & T
}\]
Since $\eta\in\Delta^n(S,T)$, we may assume that each lifting is given by
a matrix with entries in $R_0$. The lifting
$l^0$ is surjective, since $\eta$ is nonzero.  Since $\beta_n(S)=1$,
$l^0$ is also a monomorphism and therefore it is an isomorphism.  Then
it follows from Proposition~\ref{successive monos} that all the liftings
$l^j$ for $j \geqslant 1$ are monomorphisms (and thus nonzero).  By
assumption there exists some $m\geqslant 1$ such that the Betti number
$\beta_m(T)=1$.  Hence the lifting $l^m$ is surjective and therefore
all previous liftings $l^0, l^1, \dots, l^{m - 1}$ are
surjective from Proposition~\ref{successive--> epis}. Hence all the
liftings $l^0, l^1, \dots, l^{m - 1}, l^m$ are isomorphisms.  Since
$m\geqslant 1$, it follows that $\Omega^nS\cong T$.  The number of
simple modules is finite and the above is true for any simple module,
so we can conclude that all simple modules have a simple periodic
syzygy.
\end{proof}

\noindent It is possible to have an indecomposable, even a local, selfinjective algebra where all the simple modules are periodic, but their Betti numbers need not all equal 1. However, in the selfinjective {\it plus} Koszul case we have the following consequence of the previous theorem:

\begin{prop}
Let $R$ be a selfinjective Koszul algebra where all the simple
modules are periodic. Then:
\begin{enumerate}
\item[(a)] all the Betti numbers for the simple modules are $1$.
\item[(b)] $R\cong \Bbbk\widetilde{\mathbb{A}}_m/J^2$ for some
  $m\geqslant 1$, where $\widetilde{\mathbb{A}}$ has circular
  orientation and $J$ is the ideal generated by the arrows.  In
  particular, all simple modules have the same period.
\end{enumerate}
\end{prop}
\begin{proof}
(a) Let $n\geqslant 1$ be minimal such that $\Omega^nS\cong S$
for all simple modules $S$.  Let
\[\eta\colon 0\to S\to P^{n-1}_S\to \cdots \to P^1_S\to P^0_S\to
S\to 0\]
be exact with $P^i_S$ projective.  Then $\eta$ is non-nilpotent in
$\Ext^*_R(S,S)$, and we have
\begin{multline}
0\neq \eta^{t+1} = \eta^t\cdot
\underbrace{(0\to S\to P_S^{n-1}\to \cdots \to P_S^{n-i}\to \Omega^{n-i}S\to 0)}_{\theta_{n-i}}\cdot \notag\\
\underbrace{(0\to \Omega^{n-i}S\to P_S^{n-i-1}\to \cdots \to P_S^0\to S\to0)}_{\theta_i'}\notag
\end{multline}
for all $t\geqslant 0$.  In particular, $\eta^t\cdot \theta_{n-i}\neq
0$ for all $t\geqslant 1$ and for all $1\leqslant i\leqslant n-1$.  We
want to show $\eta^t\cdot \Ext^{n-i}_R(R_0,S)\neq (0)$ for some
$t\geqslant 1$.  So suppose for contradiction that $\eta^t\cdot
\Ext^{n-i}_R(R_0,S) = (0)$ for all $t\geqslant 1$.  Then we claim the
following: $\eta^u\cdot \Ext^{n-i}_R(B,S)=(0)$ for all modules $B$ of
Loewy length at most $u$.  For $u=1$ it is true by assumption.  Assume
the claim is true for $u$, and let $B$ be a module of Loewy length
$u+1$. Consider the exact sequence
\[0\to B\rrad \to B\to B/B\rrad \to 0.\]
This gives rise to the exact sequences
\[\Ext^{jn-i}_R(B/B\rrad ,S)\xrightarrow{\delta_j} \Ext^{jn-i}_R(B,S)\xrightarrow{\epsilon_j}
\Ext^{jn-i}_R(B\rrad ,S)\]
for $j\geqslant 1$.  Let $\nu$ be an element in $\Ext^{n-i}_R(B,S)$.
Then by assumption $\epsilon_{u+1}(\eta^u\cdot\nu) =\eta^u\cdot
\epsilon_1(\nu)$, so that $\eta^u\cdot \nu = \delta_{u+1}(\mu)$ for
some $\mu$ in $\Ext^{(u+1)n-i}_R(B/B\rrad ,S)\cong \Ext^{n-i}_R(B/B\rrad ,S)$,
since $\Omega^n(B/B\rrad)\cong B/B\rrad$.  Hence
$\eta\cdot\mu = 0$ and therefore $\eta^{u+1}\cdot \nu = 0$.  This
completes the proof of the above claim.  Hence, if $\eta^t\cdot
\Ext^{n-i}_R(R_0,S) = (0)$ for all $t\geqslant 1$, then
$\eta^N\cdot\Ext^{n-i}_R(X,S)=(0)$ for all finitely generated
$R$-modules $X$ where $N$ is the Loewy length of $R$. However, this is not true for
$X=\Omega_R^{n-i}S$, so this gives a contradiction.  Consequently,
given $i$ with $1\leqslant i\leqslant n-1$, there exists $t\geqslant
1$ and a simple module $T$ such that $\eta^t\cdot
\Ext^{n-i}_R(T,S)\neq (0)$.  Then we can recycle the arguments from
Proposition~\ref{prop:comparison_betti_numbers}, and we obtain that
$\beta_i(S) = 1$.  Since we can let $i$ vary freely in the interval
$[1,\ldots,n-1]$, we infer that all the Betti numbers for $S$ are $1$.
This completes the proof of the first part.

(b) Since $R_0=\Bbbk \times \cdots \times \Bbbk$ and $R$ is a Koszul algebra,
we have that $R\cong \Bbbk Q/I$ where
$Q$ is a finite quiver and $I$ is a quadratic admissible ideal in the
path algebra $\Bbbk Q$.  The fact that the two first Betti numbers are $1$
for all simple modules, implies that there is exactly one arrow
starting at each vertex in $Q$.  Since $R$ is indecomposable and all
simple modules have infinite projective dimension, the quiver $Q$ must
be $\widetilde{\mathbb{A}}_m$ for some $m\geqslant 1$ and there is a
relation starting at each vertex.  Since all relations
are quadratic, they are generated by all paths of length
$2$.  Consequently, $R\cong \Bbbk\widetilde{\mathbb{A}}_m/J^2$ for some
integer $m\geqslant 1$, where $J$ is the ideal generated by the
arrows.  It is easy to see that all the simple modules have period $m$ in
this case.
\end{proof}



\begin{thebibliography}{GGGGG}

\bibitem{BGSS} Buchweitz, R-O., Green, E.\ L., Snashall, N., Solberg,
  \O.; {\it Multiplicative structures for Koszul algebras},
  Quart.\  J.\ Math.\ {\bf 59} (2008), no.\ 4, 441--454.

\bibitem{CR}  Curtis, C.\ W., Reiner, I., \emph{Methods of representation
  theory},  Vol.\ II. With applications to finite groups and orders. Pure
  and Applied Mathematics (New York). A Wiley-Interscience
  Publication. John Wiley \& Sons, Inc., New York, 1987. xviii+951
  pp. ISBN: 0-471-88871-0.

\bibitem{D1} Dugas, A.\ S.; \emph{On periodicity in bounded projective
    resolutions}. arXiv:1203.2408

\bibitem{D2} Dugas, A.\ S.; \emph{Periodic resolutions and
    self-injective algebras of finite type}. J.\ Pure Appl.\ Algebra
  {\bf 214} (2010), no.\ 6, 990--1000.

\bibitem{EHSST} Erdmann, K., Holloway, M., Snashall, N., Solberg, \O.,
  Taillefer, R.; \emph{Support varieties for selfinjective algebras},
  K-Theory {\bf 33} (2004), no.\ 1, 67--87.

\bibitem{GHMS} Green, E.\ L., Hartman, G., Marcos, E.\ N., Solberg,
  \O.; {\it Resolutions over Koszul algebras}, Archiv der Math.\ {\bf
    85} (2005), no.\ 2, 118--127.

\bibitem{GSS} Green, E.\ L., Snashall, N., Solberg, \O.; \emph{The
    Hochschild cohomology ring of a selfinjective algebra of finite
    representation type}, Proc.\ Amer.\ Math.\ Soc. {\bf 131} (2003),
  3387--3393.

\bibitem{S} Snashall, N.; {\em Support varieties and the Hochschild
    cohomology ring modulo nilpotence.} Proceedings of the 41st
  Symposium on Ring Theory and Representation Theory, 68--82,
  Symp. Ring Theory Represent. Theory Organ. Comm., Tsukuba, 2009.

\bibitem{SS} Snashall, N., Solberg \O.; {\it Support varieties and
    Hochschild cohomology rings}. Proc. London Math. Soc. (3) {\bf 88}
  (2004), no.\ 3, 705--732.

\bibitem{X} Xu, F.; {\it Hochschild and ordinary cohomology rings of
    small categories}.  Adv.\ Math.\ {\bf 219} (2008), no.\ 6, 1872--1893.
\end{thebibliography}
\end{document}